\documentclass[12pt,letterpaper]{amsart}
\usepackage[makeindex]{imakeidx}
\usepackage{ amsmath,amsthm, amsbsy,amsfonts,amssymb, txfonts}
\usepackage[normalem]{ulem}
\usepackage{stmaryrd,mathrsfs,graphicx, supertabular, amscd, tikz-cd, tikz}
\usepackage{stackrel}
\usetikzlibrary{matrix,arrows,decorations.pathmorphing}

\usepackage[active]{srcltx}
\allowdisplaybreaks
\numberwithin{equation}{section}

\usepackage[
	hypertexnames=false,
	hyperindex,
	pagebackref,
	breaklinks=true,
	bookmarks=false,
	colorlinks,
	linkcolor=blue,
	citecolor=red,
	urlcolor=red,
]{hyperref}
\usepackage{hyperref}

\def\RR{{\mathbb R}}

\def\ZZ{{\mathbb Z}}

\def\ssm{\smallsetminus}

\def\g{\gamma}

\def\int{{\rm int}}

\newcommand{\eps}{\varepsilon}

\newcommand{\p}{\partial}

\def\Ical{{\mathcal I}}

\def\half{{\tfrac{1}{2}}}

\def\pfrak{\mathfrak{p}}

\def\pt{{\scriptscriptstyle\bullet}}

\newcommand\gr{\operatorname{gr}}

\newcommand\Hom{\operatorname{Hom}}

\newtheorem{theorem}{Theorem}[section]
\newtheorem{lemma}[theorem]{Lemma}

\newtheorem{corollary}[theorem]{Corollary}

\theoremstyle{remark}
\newtheorem{example}[theorem]{Example}

\newtheorem{remark}[theorem]{Remark}

\title[Motivic fundamental group rings]{On the motivic description of truncated fundamental group rings}

\begin{document}
\author{Eduard Looijenga}\thanks{Research for this paper was done when the author was supported by the Jump Trading Mathlab Research Fund}
\address{Mathematics Department, University of Chicago (USA) and Mathematisch Instituut, Universiteit Utrecht (Nederland)}
\email{looijenga@uchicago.edu, e.j.n.looijenga@uu.nl}


\begin{abstract}
A topological theorem that appears in a paper by Deligne-Goncharov \cite{dg} states the following. Let  $(X,*)$ be a path connected  pointed space with a  reasonable topology and denote by $\Ical$ the augmentation ideal  of the group ring $\ZZ\pi_1(X,*)$. Then for every field $F$ and positive  integer $n$,  the space of $F$-valued linear forms on  $\Ical/\Ical^{n+1}$ is naturally isomorphic to $H^n(X^n,\text{a certain subspace of }X^n; F)$. 

Among other things we give a simple construction of  an isomorphism between 
$\Ical/\Ical^{n+1}$ and the integral $H_n$ of this pair and  express the maps that define the Hopf algebra structure on $\ZZ\pi_1(X,*)/\Ical^{n+1}$ in these terms.
\end{abstract}

\maketitle

\section*{Introduction}
This note arose from a desire to better understand a theorem in a paper by Deligne-Goncharov (\cite{dg}, Prop.\ 3.4), who attribute this to  Beilinson. The theorem (and its proof) is of a topological nature and  easy to state. 
In the version that we discuss here it amounts to the   following. 

Let $X$ be  a  connected topological space with a reasonable topology (we are more specific about this below) and let $a, b\in X$. Consider the set $\pi_X(a,b)$ of homotopy classes of  paths in $X$ from $a$ to $b$. Since this is a principal right $\pi_1(X,a)$-set, the free abelian group  
$\ZZ\pi_X(a,b)$ it generates  is a principal  right $\ZZ\pi_1(X,a)$-module (it is also a principal  left $\ZZ\pi_1(X,b)$-module). So if $\Ical_a\subset \ZZ\pi_1(X,a)$ denotes the augmentation ideal, then the quotient module $\ZZ\pi_X(a,b)/\ZZ\pi_X(a,b)\Ical_a^{n+1}$  is a principal right module for the truncated group ring  $\ZZ\pi_1(X,a)/\Ical_a^{n+1}$. Letting $X(n)^a_b\subset X^n $ stand for  the set of $(x_1, \dots, x_n)\in X^n$ for which one of the following identities $x_1=a$, $x_\nu=x_{\nu +1}$ ($\nu=1, \dots , n-1$) or $x_n=b$ holds, then the  theorem identifies the  relative homology group $H_n(X^n,X(n)^a_b)$ with $\ZZ\pi_X(a,b)/\ZZ\pi_X(a,b)\Ical_a^{n+1}$  if $b\not=a$ and with $\Ical_a/\Ical_a^{n+1}$ if $b=a$.

This has immediate  applications in algebraic geometry, as it gives us a motivic interpretation  of the  abelian groups  $\ZZ\pi_X(a,b)/\ZZ\pi_X(a,b)\Ical_a^{n+1}$  when  $X$ is a connected complex algebraic variety. In particular, this provides an alternate way to define mixed Hodge structures on these (originally due to Hain \cite{hain:pi1} and based on the reduced bar complex). Our motivation is  however different, for we expect this theorem (and its proof ingredients) to help us understand the way a mapping class group acts on the cohomology of the  configuration spaces of a surface. 

The result as stated above strengthens that in \cite{dg} (where the dual version is proved with field coefficients). 
We give in addition a simple  explicit  description of the map from  $\ZZ\pi_X(a,b)/\ZZ\pi_X(a,b)\Ical_a^{n+1}$ 
to $H_n(X^n,X(n)^a_b)$. We complete the story by verifying that the structural  maps involving the abelian groups 
$\ZZ\pi_X(a,b)/\ZZ\pi_a(X)\Ical_a^{n+1}$ (truncation, inversion, composition,  and in case $a=b$, the coproduct) have homologically defined  counterparts. This last property guarantees  that when $X$ is a connected complex algebraic variety, these structural map are morphisms of mixed Hodge structures. Somewhat  surprisingly, the composition map, which here  takes the form of Theorem \ref{thm:composition}, turned out to be the least trivial in this respect.

Where the proof of the Beilinson-Deligne-Goncharov is concerned, ours is in the end probably not so different from theirs. Yet we feel that the strengthening presented here and its proof  make the discussion more direct (if not shorter) and for a topologist perhaps also more transparent.  We also pay attention to homotopy aspects.

\smallskip
I thank Sidanth Venka Raman for carefully reading an earlier version of this paper. His questions prompted me to recast some of the arguments.

\smallskip
Throughout this note  $X$ stands for a  connected, compactly generated topological space which admits a CW-structure.
This guarantees that a cell decomposition of $X$ determines one of $X^n$ and that a projection on a factor is cellular.
When $n\ge 2$, we can always refine the product cell decomposition of $X^n$ such that each diagonal locus is a subcomplex and the projection onto a factor is still cellular.

\section{The isomorphism of Beilinson-Deligne-Goncharov} 
Recall that the fundamental groupoid $\pi_X$ of $X$ has  object set $X$  with for $a,b\in X$  the set $\pi_X(a,b)$ of homotopy classes of paths from $a$ to $b$ as its hom set. We abbreviate  
 $\pi_a$ for  $\pi_1(X,a)$ and write  $\Ical_a\subset \ZZ\pi_a$ for the augmentation ideal of the group ring $\ZZ\pi_a$.  We regard  $ \ZZ\pi_X(a,b)$ as a principal left $\ZZ\pi_b$-module and as   principal right $ \ZZ\pi_a$-module.  Note that $\Ical_b^{n+1} \ZZ\pi_X(a,b)= \ZZ\pi_X(a,b)\Ical_a^{n+1}$; we shall write $\ZZ\pi_X(a,b)_n$ for the quotient
 of $\ZZ\pi_X(a,b)$ by this submodule.

The map which assigns to an element of $\pi_X(a,b)$ its image in 
$H_1(X,\{a,b\})$ induces a homomorphism $ \ZZ\pi_X(a,b)/\Ical_a^2\to H_1(X,\{a,b\})$. This is an isomorphism 
if $a\not=b$ and is surjective with kernel $\ZZ 1_a$ if $a=b$. We shall give an interpretation of $ \ZZ\pi_X(a,b)/\Ical_a^{n+1}$
that generalizes this observation.
\\

Since we use the categorical convention for which composition of paths is read from right to left, it is convenient to index the factors of $X^n$ in opposite order, so that  a typical point of $X^n$ is denoted $(x_n, \dots, x_1)$.

For an  integer $n\ge 1$ and  a pair $(a,b)\in X^2$, we define subsets of $X^n$ by 
\begin{gather*}
X(n)^a:=(x_1=a) \cup (x_1=x_2)\cup\cdots\cup  (x_{n-1}=x_{n}),\\
X(n)_b:=(x_1=x_2)\cup\cdots\cup (x_{n-1}=x_{n})\cup (x_n=b),\\
X(n)_b^a:=(x_1=a) \cup (x_1=x_2)\cup\cdots\cup (x_{n-1}=x_{n})\cup (x_n=b),
\end{gather*}

Let $\triangle^n$ stand for the $n$-simplex in $\RR^n$ that consists of the $(t_n,\dots, t_1)\in \RR^n$ defined by $0\le t_1\le t_2\le\cdots \le t_n\le 1$ (we stipulate that  $\triangle^0=\{0\}$). 
If $\g :[0,1]\to X$ is a path from $a$ to $b$, then  $\g^n: [0,1]^n\to X^n$ maps $\triangle^n$ to $X^n$ with $\p\triangle^n$ mapping to 
$X(n)^a_b$  and hence defines a class in $H_n(X^n,X(n)^a_b)$. It is clear that this relative homology class only depends on the relative homotopy class $[\g]\in \pi_X(a,b)$ of $\g$. We therefore denote that class by 
$\triangle^n[\g]$. We thus have defined a homomorphism of abelian groups
\begin{equation}\label{eqn:motivic1}
\tilde\kappa_n:  \ZZ\pi_X(a,b)\to H_n(X^n,X(n)^a_b).
\end{equation}
It is clear that the passage from $\g$ to its inverse $\g^{-1}$ defines an isomorphism of $\ZZ\pi_X(a,b)\to \ZZ\pi_X(bra)$ of abelian groups. Its homological counterpart  is given  by the  order reversal  that yields the isomorphism 
$H_n(X^n,X(n)^a_b)\to H_n(X^n,X(n)^b_a)$.

Here is our version of a theorem in a paper by  Deligne-Goncharov, who, as mentioned,  attribute this to Beilinson. 

\begin{theorem}[Dual, integral version of \cite{dg}, Prop.\ 3.4] \label{thm:motivic}
The pair $(X^n,X(n)^a_b)$ is  $(n-1)$-connected and the 
map $\tilde\kappa_n$ in \eqref{eqn:motivic1} factors through a surjection 
\begin{equation}\label{eqn:motivic}
\kappa_n:  \ZZ\pi_X(a,b)_n\to H_n(X^n,X(n)^a_b)
\end{equation}
whose kernel is trivial unless  $a=b$, in which case it equals $\ZZ 1_a$.
\end{theorem}
\begin{proof}
We already observed that  the theorem  holds for $n=1$. We proceed  with induction on $n$ and  assume $n\ge 2$. Let $f:(X^n,X(n)^a_b)\to X$  denote  the projection on the last factor.  

\vskip2mm

\emph{Step 1: $(X^n,X(n)^a_b)$ is  $(n-1)$-connected.}

The fiber pair of $f$ over $x\in X$ is  a copy of $(X^{n-1},X(n)^a_x)$ unless 
$x=b$, in which case the fiber pair is a copy of $(X^{n-1},X^{n-1})$. 
By our induction hypothesis, these fiber pairs  are $(n-2)$-connected.
The assertion  is a formal consequence of this property and the fact that $(X,\{b\})$ is $0$-connected.	
\vskip2mm

Choose a CW structure on $X$ and $X^n$ for which $f$ is a cellular map and $\{a,b\}$ resp.\  $X(n)^a_b$  are 
subcomplexes. By possibly further refinement, we can assume that for every $0$-cell $\{c\}$ distinct from $\{b\}$  there is a $1$-cell $e_c$ of $X$  connecting  $b$ and $c$ (\footnote{If $X$ is a topological manifold, we can in fact arrange that $a$ and $b$ (or even just $b$) are the only $0$-cells. This would simplify the subsequent discussion a bit.}). We  orient  all the cells of positive  dimension and assume that  $e_c$ goes from $b$ to $c$. We write $E_k$ for the collection of $k$-cells of $X$ and denote by $E_1(b)\subset E_1$ the collection $\{e_c\}_c$ just defined.

Let $(\hat X, \hat a)\to (X,a)$ be the universal cover of $(X,a)$: a point of $\hat X$ over $x\in X$ is specified by an element of $\pi_X(a,x)$. The CW structure on $X$ lifts to one on $\widehat X$: a cell of  $\hat X$ over a given cell $e$ of $X$ is specified by an element of $\pi_{X}(a,x)$ for some $x\in e$. Since $e$ is simply connected, this is independent of $x$ and so we may  just as well denote this set (a right torsor of $\pi_a$) by $\pi_X(a,e)$. 
This identifies the augmented cellular chain complex of  $\hat{X}$ with the complex of right $\ZZ\pi_a$-modules
\begin{equation}\label{eqn:piresolution}
\cdots\to \oplus_{e\in E_2} \ZZ\pi_X(a,e)\xrightarrow{\p_2} \oplus_{a\in E_1} \ZZ\pi_X(a,e)\xrightarrow{\p_1} \oplus_{a\in E_0} \ZZ\pi_X(a,e) \xrightarrow{\eps}  \ZZ\to 0,
\end{equation}
where $\eps$ takes the value $1$ on $\cup_{a\in E_0}\pi_X(a,e)$. 
The cover $\hat{X}$ is simply connected and so this  complex is exact in homological degree $\le 1$. Since $\eps$  is an isomorphism when restricted to the first summand of $\ZZ\pi_X(a,a)=\ZZ\pi_a =\ZZ\oplus\Ical_a$,  the  quotient 
\begin{equation}\label{eqn:Iresolution2}
K_\pt: 0\to\oplus_{e\in E_2} \ZZ\pi_X(a,e)\xrightarrow{\p_2} \oplus_{e\in E_1} \ZZ\pi_X(a,e)\to \oplus_{e\in E_0\ssm \{a\}} \ZZ\pi_X(a,e)\to 0.
\end{equation}
of \eqref{eqn:piresolution} is exact  in degree zero and has  homology in degree 1 equal to $\Ical_a$.

Alternatively, the subcomplex 
\[
K'_\pt: 0\to \oplus_{e\in E_1(b)} \ZZ\pi_X(a,e)\xrightarrow{\p_1} \oplus_{e\in E_0} \ZZ\pi_X(a,e) \xrightarrow{\eps}  \ZZ\to 0
\]
of the complex \eqref{eqn:piresolution} has homology equal to $\Ical_a$ (in degree $0$), so that the quotient complex 
\begin{equation}\label{eqn:presI}
K_\pt/K'_\pt: 0\to \oplus_{e\in E_2} \ZZ\pi_X(a,e)\xrightarrow{\p'_2} \oplus_{e\in E_1\ssm E_1(b)} \ZZ\pi_X(a,e)\to 0
\end{equation}
has homology in degree $1$ (which is the cokernal of $\p_2'$) also equal to $\Ical_a$.
\vskip2mm

\emph{Step 2: Proof that $\tilde\kappa_n$ induces the surjection $\kappa_n$ with kernel $\ZZ 1_a$ when $b=a$.}

Since $f$ is cellular,  our induction hypothesis implies that the  preimage of a cell $e$ of $X\ssm \{a\}$ gives a topological pair that is $(n-2)$-connected and has its  homology in degree $n-1$ identified with $\ZZ\pi_X(a,e)_{n-1}$. A standard  spectral sequence argument then shows that  $H_n(X^n,X(n)^a_a)$ is the homology in degree one of the  complex of $\ZZ\pi_a$-modules
\[
C_\pt : 0\to\oplus_{e\in E_2}\ZZ\pi_X(a,e)_{n-1}\to \oplus_{e\in E_1}\ZZ\pi_X(a,e)_{n-1}\to  \oplus_{e\in E_0\ssm \{a\}}\ZZ\pi_X(a,e)_{n-1}\to 0.
\]
So it remains to identify $H_1(C_\pt)$ with $\Ical_a/\Ical_a^{n+1}$. Observe  that  $C_\pt\cong K_\pt\otimes_{\ZZ\pi_a} \ZZ\pi_a/\Ical_a^{n}$ and so  $H_1(C_\pt)\cong \Ical_a\otimes_{\ZZ\pi_a} (\ZZ\pi_a/\Ical_a^n)$.  To complete the argument, consider the exact sequence of $ \ZZ\pi_a$-modules 
\[
0\to \Ical_a^{n}\to \ZZ\pi_a\to \ZZ\pi_a/\Ical_a^n\to 0.
\]
This remains exact if  we apply the right exact functor $\Ical\otimes_{\ZZ\pi_a}$:
\[
\Ical_a\otimes _{\ZZ\pi_a} \Ical_a^{n}\to \Ical_a\to \Ical_a\otimes _{\ZZ\pi_a}(\ZZ\pi_a/\Ical_a^n)\to 0.
\]
The image of the first map is $\Ical_a^{n+1}\subset \Ical_a$ and so $\Ical_a\otimes _{\ZZ\pi_a}\ZZ\pi_a/\Ical_a^n\cong \Ical_a/\Ical_a^{n+1}$.
We leave it to the reader to check that our induction hypothesis also shows that the isomorphism 
$\ZZ\pi_{a,n}/\ZZ1_a \cong \Ical_a/\Ical_a^{n+1}\cong H_n(X^n,X(n)^a_a)$ thus obtained is the one defined by $\kappa_n$.
\vskip2mm

\emph{Step 3: Proof that $\tilde\kappa_n$ induces the isomorphism $\kappa_n$ when $b\not= a$.}

The proof is similar to that of step 1. The preimage of a cell $e$ of $X$ is a topological pair that is still $(n-2)$-connected, the difference  is  that  our induction hypothesis now tells us that its homology in degree $n-1$ is  
$\pi_a/(\ZZ 1_a+\Ical_a^n)$ when $e=\{a\}$,  trivial when $e=\{b\}$, and $\ZZ\pi_X(a,e)_{n-1}$ otherwise.  So if we proceed as in step 1, then we find that  $H_n(X^n,X(n)^a_b)$  is now identified with $H_1$ of the complex 
\begin{multline*}
B_\pt: 0\to \oplus_{e\in E_2}\ZZ\pi_X(a,e)_{n-1}\to  \oplus_{a\in E_1}\ZZ\pi_X(a,e)_{n-1}\to \\
\to \oplus_{e\in E_0\ssm \{b\}}\ZZ\pi_X(a,e)_{n-1}/\ZZ 1_a\to 0.
\end{multline*}
This contains the subcomplex
\[
B'_\pt: 0\to \oplus_{\{c\}\in E_0\ssm \{b\}}\ZZ\pi_X(a,e_c)_{n-1}\to\oplus_{\{c\}\in E_0\ssm \{b\}}\ZZ\pi_X(a,c)_{n-1}/\ZZ 1_a\to 0.
\]
of which the middle map is given by $\sum_{c\in E_0(b)} e_cu_c \mapsto \sum_{c\not=b} (c)u_c \mod{\ZZ 1_a}$, 
where  $u_c\in \ZZ\pi_a/\Ical_a^n$. So this map is onto with kernel $\ZZ e_a$. The quotient complex has the form 
\begin{equation}\label{eqn:B/B'}
B_\pt/B'_\pt: 0\to \oplus_{e\in E_2}\ZZ\pi_X(a,e)_{n-1}\xrightarrow{\p}  \oplus_{e\in E_1\ssm E_1(b)}\ZZ\pi_X(a,e)_{n-1}\to 0,
\end{equation}
which is just the complex \eqref{eqn:presI} tensored with $\ZZ\pi_a/\Ical_a^n$.
Its homology in degree $1$ (the cokernel of the above map) is therefore 
$\Ical_a\otimes \ZZ\pi_a/\Ical_a^n$, which we showed, can be identified  with $\Ical_a/\Ical_a^{n+1}$.
So the long   exact sequence associated to the pair $(B_\pt, B'_\pt)$ is of the form
\begin{equation}\label{eqn:noncan}
0\to H_2(B_\pt)\to H_2(B_\pt/B'_\pt)\to \ZZ e_a\to H_n(X^n,X(n)^a_b)\to  \Ical_a/\Ical_a^{n+1}\to 0.
\end{equation}
No multiple of $e_a$ is a boundary of $B_\pt$. So  $e_a\ZZ\to H_n(X^n,X(n)^a_b)$ is injective. This makes
$H_n(X^n,X(n)^a_b)$ appear as an extension of $\Ical_a/\Ical_a^{n+1}$ by $\ZZ e_a$. 

We make this extension explicit in terms of $\kappa_n$. By our induction hypothesis, a typical generator of $H_{n-1}(X^{n-1},X(n-1)^a_a)$ is of the form $\triangle^{n-1}[\g]$, where $\g: [0,1]\to X$ is a loop based at $a$. Let $\g_o:[0,1]\to X$ be a parametrization of the closure of the $1$-cell $e_a$ with $\g_o(0)=a$ and $\g_o(1)=b$. The composite $\g_o\g$ is a path from $a$ to $b$ with  the property that  $\triangle^n[\g_o\g]\in H_{n}(X^{n},X(n)^a_a)$  restricted to $t_n=1$ gives us $\triangle^{n-1}[\g]$. This defines a section of the map $H_n(X^n,X(n)^a_b)\to  \Ical_a/\Ical_a^{n+1}$ constructed above. This also shows that $H_n(X^n,X(n)^a_b)$ is naturally identified with
$\ZZ\pi_X(a,b)_n$ (this also shows that  the sequence \eqref{eqn:noncan} cannot be canonical; indeed,  it involves the $1$-cell $e_a$).
\end{proof}

\begin{remark}\label{rem:}
It is clear from the definition that $\kappa_n$ naturally lifts to a homomorphism  $\ZZ\pi_X(a,b)/\Ical_{a}^{n+1}\to\pi_n(X^n,X(n)^a_b)$. Theorem \ref{thm:motivic} and the   Hurewicz theorem  imply that  this has also kernel $\ZZ 1_a$ when $a=b$ and  is an isomorphism (when $a\not=b$), provided that $n\ge 2$.
In particular,  for $n\ge 2$, the Postnikov system of the pair $(X^n, X(n)^a_b)$ begins properly at stage $n$: it is then the 
Eilenberg-MacLane space $K(G,n)$, where  $G$ equals  $\Ical_a/\Ical_a^{n+1}$ or $\ZZ\pi_b^a/\ZZ\pi_b^a\Ical_a^{n+1}$ depending on whether or not $a=b$.
We may therefore  regard this as  the `motivic homotopy incarnation'  of  that group.
\end{remark}

\begin{remark}\label{rem:}
The covering of $X(n)^a_b$  by its closed subsets $(x_1=a), (x_1=x_2), \dots , (x_{n-1}=x_n), (x_n=b)$ has a 
nerve which is the boundary of an $(n+1)$-simplex  (so a topological $S^n$) when   $a\not= b$ and is
a cone  when $a=b$ (for then their common  intersection is $(a,a,\dots ,a)$ and  hence nonempty). This accounts for the change of $\ker(\kappa_n)$ in these two cases.
\end{remark}

\section{The coproduct}\label{sect:coproduct} The coproduct $\ZZ\pi_a\to \ZZ\pi_a\otimes\ZZ\pi_a$ comes from the group homomorphism 
$\pi_a=\pi_1(X,a)\to \pi_1(X^2, a^2)=\pi_a\times\pi_a$ induced by the diagonal embedding $D: X\to X^2$. We identify the group ring $\ZZ(\pi_a\times\pi_a)$ with  $\ZZ\pi_a\otimes\ZZ\pi_a$, so that its augmentation ideal
is $\Ical_a\otimes \ZZ\pi_a+\ZZ\pi_a\otimes \Ical_a$.
This group homomorphism gives a  map 
\begin{multline*}
\Ical_a/\Ical_a^{n+1}\to (\Ical_a\otimes \ZZ\pi_a+\ZZ\pi_a\otimes \Ical_a)/(\Ical_a\otimes \ZZ\pi_a+\ZZ\pi_a\otimes \Ical_a)^{n+1}=\\
\textstyle =(\Ical_a\otimes \ZZ\pi_a+\ZZ\pi_a\otimes \Ical_a)/(\sum_{i+j=n+1} \Ical_a^i\otimes \Ical_a^j). 
\end{multline*}
Via Theorem \ref{thm:motivic} this is the map $H_n(X^n,X(n)^a_a)\to H_n(X^{2n},X^2(n)^{a^2}_{a^2})$ induced by $D$.  So if we recall that the primitive part $\pfrak(H)$ of a bi-algebra $H$ consists of the $a\in H$  that under comultiplication are mapped to $a\otimes 1+1\otimes a$, we find: 

\begin{corollary}\label{cor:}
The primitive part $\pfrak(\ZZ\pi_a/\Ical_a^{n+1})$ of $\ZZ\pi_a/\Ical_a^{n+1}$  maps via $\kappa_n$ isomorphically onto the
equalizer of the pair 
\[
(D_*, i'_*+i''_*) :H_n(X^n,X(n)^a_a)\rightrightarrows  H_n(X^{2n},X^2(n)^{a^2}_{a^2}),
\]
where $i', i'': (X, a)\hookrightarrow (X^2,a^2)$ are defined by $x\mapsto (x,a)$ resp.\  $x\mapsto (a,x)$.
\end{corollary}

It would of course be better to have a  more direct characterization of  the primitive part.

\section{The truncation map}\label{sect:trunc}
For $n\ge 2$, the embedding  $(x_{n-1},\dots, x_1)\in X^{n-1}\mapsto (b, x_{n-1},\dots, x_1)\in X^n$ defines a relative homeomorphism $(X^{n-1},X(n-1)^a_b)\to (X(n)^a_b,X(n)^a)$. If we substitute this in homology exact sequence of the triple  $(X^n,X(n)^a_b, X(n)^a)$ we find the exact fragment
\begin{multline*}
H_{n}(X^{n-1},X(n-1)^a_b)\to H_n(X^n,X(n)^a)\to\\
\to  H_n(X^n,X(n)^a_b)\to H_{n-1}(X^{n-1},X(n-1)^a_b)
\end{multline*}
We denote the last map by $\tau^n_{n-1}$. 
Here we also include the case $n=1$, by stipulating  that
$H_0(X^0,X(0)^a_b)=H_{0}(\{a,b\},\{a\})$ (which is trivial when $a=b$ and $\ZZ$ otherwise). We  shall interpret each $\tau^n_{n-1}$ as a truncation map. 

We denote  cokernel of $H_{n}(X^{n-1},X(n-1)^a_b)\to H_n(X^n,X(n)^a)$  by $C_n(X)^a_b$ so that it can be identified with the kernel of $\tau^n_{n-1}$. We shall later see (Lemma \ref{lemma:motivic_fg}) that the natural map 
$H_1(X)^{\otimes n}\to H_n(X^n,X(n)^a)$ is an isomorphism, so that $C_n(X)^a_b)$ then appears as a quotient of
$H_1(X)^{\otimes n}$. This quotient turns out to be independent of $a$ and $b$ and is identified in the corollary below. For this we recall that the subquotient $\Ical_a^n/\Ical_a^{n+1}$ of $\ZZ\pi_a$  is  independent of $a$ (for inner automorphisms of $\pi_a$ acts trivially on it) and  functorial in $X$. We denote this subquotient by $S_n(X)$, so that $S_\pt(X)=\gr_{\Ical_a}\ZZ\pi_a$ is the graded algebra associated to the $\Ical_a$-adic filtration of $\ZZ\pi_a$.

\begin{corollary}\label{cor:motivic_fg}
For  $n\ge 1$, we have a commutative diagram of short exact sequences
\begin{center}
\begin{small}
\begin{tikzcd}
0 \arrow[r] & S_n(X)\arrow[r]\arrow[d, "\cong"] &  \ZZ\pi_X(a,b)/\Ical_a^{n+1} \arrow[r, "truncation"]\arrow[d, "\kappa_n"]  &  \ZZ\pi_X(a,b)/\Ical_a^{n}\arrow[r]\arrow[d, "\kappa_{n-1}"] & 0\\
0 \arrow[r] & C_n(X)^a_b \arrow[r] &H_n(X^n,X(n)^a_b) \arrow[r, "\tau^n_{n-1}"] & H_{n-1}(X^{n-1},X(n-1)^a_b)\arrow[r] &  0
\end{tikzcd}
\end{small}
\end{center}
In particular, truncation maps $\ker(\kappa_n)$ isomorphically onto $\ker(\kappa_{n-1})$.
\end{corollary}
\begin{proof}
Our conventions ensure that this is  true for $n=1$.
For $n>1$ the corollary boils down to the commutativity of the right hand square. 
This follows the very definition of the map $\kappa_n$: given a path $\g$ in $X$ from $a$ to $b$, then $\tau^n_{n-1}\kappa_n$  takes the image $[\g]_n$ of 
$\g$ in $\ZZ\pi_X(a,b)_n$ to the restriction of $\triangle_n[\g^n]$ to the face $t_n=1$ and this is just  
$\triangle_{n-1}[\g^{n-1}]$, that is, the image of $[\g]_{n-1}$ under $\kappa_{n-1}$.
\end{proof}

\begin{remark}\label{rem:connectedness}
So the projective limit of the system
\begin{equation}\label{eqn:proj}
\textstyle \cdots \xrightarrow{\tau}H_n(X^n, X(n)^a_a)\xrightarrow{\tau} H_{n-1}(X^{n-1}, X(n-1)^a_a)\xrightarrow{\tau}  \cdots
\end{equation}
gives us the augmentation ideal $\hat\Ical_a=\varprojlim_n \Ical_a/\Ical_a^n$ of  $\hat\pi_a:=\varprojlim_n \pi_a/\Ical_a^n$. 
Since $H_n(X^n, X(n)^a_a)$ is $(n-1)$-connected, the universal coefficient theorem tells us that the dual of $H_n(X^n, X(n)^a_a)$ is $H^n(X^n, X(n)^a_a)$.
Hence $\varinjlim_n H^n(X^n, X(n)^a_a)$ gives us the continuous dual $\varinjlim_n \Hom (\Ical_a/\Ical_a^n, \ZZ)$ of $\hat\Ical_a$.
\end{remark}

\begin{remark}\label{rem:}
The system $\{\kappa_n\}_n$  is evidently functorial in the homotopy category of triples  $(X;a,b)$ of the type considered here: any continuous map 
$f:(X;a,b)\to (X';a',b')$ in our category induces a map of  commutative ladders which only depends on the relative homotopy class of $f$.
In particular, the group of self-homotopy equivalences of $(X; a,b)$ acts on the projective system \eqref{eqn:proj}.  The inner automorphisms of $\pi_a$ acting on $\ZZ\pi_b^a/\Ical_a^{n+1}$ are thus realized by `point pushing' of $a$ along a loop in $X$ (take the loop in $X\ssm\{b\}$ when $b\not=a$); in case if $X$ is a topological manifold, these are in fact realized  by homeomorphisms of $(X;a,b)$.
\end{remark}

\begin{remark}\label{rem:}
Note that $S_\pt(X)$  has the structure of a graded Hopf algebra. We have $S_\pt(X^2)=S_\pt(X)\otimes S_\pt(X)$.
 and it  follows from  the discussion  in Section \ref{sect:coproduct} that its primitive part in degree $n$ is identified with the equalizer of $(D_*, i'_*+i''_*) :C_n(X)^a_b\rightrightarrows  C_n(X^2)^{a^2}_{b^2}$. 
 
 For $\alpha\in \pi_a$, we have $a:=\alpha-1\in \Ical_a$ and the coproduct takes $a$ to 
$a\otimes 1+1\otimes a+a\otimes a\equiv a\otimes 1+1\otimes a\pmod{\Ical^2_a}$.
The module $S_n(X)$ is spanned by monomial expressions $a_n a_{n-1}\cdots a_1$ with $a_i\in H_1(X)$ and 
the coproduct takes such an expression to 
\[
\textstyle (a_n\otimes 1+1\otimes a_n)\cdots (a_1\otimes 1+1\otimes a_1)=\sum_I a_I\otimes a_{I'},
\]
where the sum is over all subsets  $I\subset \{n,\dots, 1\}$ (with the induced order, so if $I=\{i_k>i>\cdots >i_1\}$, then
$a_I=a_{i_k}\cdots a_{i_1}$) and $I'=\{n, \dots, 1\}\ssm I$.

When  $\pi_a$ is free, then it is known that $H_1(X)^{\otimes n}\to \Ical_a^n/\Ical_a^{n+1}$ is an isomorphism. So  $S_\pt(X)$ is then the tensor algebra on $H_1(X)$. 
\end{remark}

\section{The composition map}\label{sect:comp} 
Given $a,b,c\in X$, then path composition $\pi_X(b,c)\times \pi_X(a,b)\to \pi_X(a,c)$ induces 
a bilinear map $\ZZ\pi_X(b,c) \times \ZZ\pi_X(a,b)\to\ZZ \pi_X(a,c)$ (this  in fact descends to an  isomorphism
$\ZZ\pi_X(b,c) \otimes_{\ZZ\pi_b} \ZZ\pi_X(a,b)\cong\ZZ \pi_X(a,c)$). By reducing modulo $\Ical_b^{n+1}$ we 
find  a bilinear map
\[
\ZZ\pi_X(b,c)_n\times \ZZ \pi_X(a,b)_n\to \ZZ \pi_X(a,c)_n.
\]
In the special case when $a=b=c$, we can do slightly better, for then path composition defines  a natural bilinear map
\[
\Ical_a/\Ical_a^{n}\times \Ical_a/\Ical_a^{n}\to \Ical_a^2 /\Ical_a^{n+1}.
\]
In view of Theorem \ref{thm:motivic}, these compositions  correspond to bilinear maps
\begin{equation}\label{eqn:comp1}
H_n(X^n, X(n)^b_c)\times H_n(X^n, X(n)^a_b)\to H_n(X^n, X(n)^a_c), \text{  resp.}\\
\end{equation}
\begin{equation}\label{eqn:comp2}
H_{n-1}(X^{n-1}, X(n-1)^a_a)\times H_{n-1}(X^{n-1}, X(n-1)^a_a)\to H_{n}(X^{n}, X(n)^a_a).
\end{equation}
Our goal is to explicate these maps. The case $n=1$ is instructive, as it shows that even then the issue is not that simple:

\begin{example}\label{example:}
The map \eqref{eqn:comp2} is then without interest, as it is represented by the trivial map  $0\times 0\to H_1(X,\{a\})$. 
For the  \eqref{eqn:comp1}, we note that $H_1(X,\{b,a\})$ contains $H_1(X)$ with quotient $\tilde H_0(\{b,a\})\cong \ZZ(b-a)$ and similarly for
 $H_1(X,\{c,b\})$. The quotient of $H_1(X,\{c,b\})\otimes H_1(X,\{b,a\})$ by $H_1(X)\otimes H_1(X)$ fits in a cartesian diagram

\hskip-7mm
\begin{tikzcd}
H_1(X,\{c,b\})\otimes H_1(X,\{b,a\})/H_1(X)\otimes H_1(X)\arrow[r] \arrow[d] &
\tilde H_0(\{c,b\})\otimes H_1(X,\{b,a\})\arrow[d] \\
H_1(X,\{c,b\})\otimes \tilde H_0(\{a,b\}) \arrow[r] &\tilde H_0(\{c,b\})\otimes\tilde H_0(\{a,b\})
\end{tikzcd}
 \noindent We regard $\tilde H_0(\{c,b\})\otimes H_1(X,\{b,a\})$ and $H_1(X,\{c,b\})\otimes \tilde H_0(\{a,b\})$
as $H_1(X)$-torsors over $\tilde H_0(\{c,b\})\otimes\tilde H_0(\{a,b\}$ and then take the fiber product as torsors:
if $u,u'\in \tilde H_0(\{c,b\})\otimes H_1(X,\{b,a\})$ and $v, v'\in H_1(X,\{c,b\})\otimes \tilde H_0(\{a,b\})$ lie over
$w\in \tilde H_0(\{c,b\})\otimes\tilde H_0(\{a,b\}$, then we identify $(u,v)$ with $(u',v')$ if and only if $u-u'=v-v'$.
The resulting quotient  of $H_1(X,\{c,b\})\otimes H_1(X,\{b,a\})$ can be identified with $H_1(X,\{a,c\})$.
\end{example}

This  generalizes as follows.

\begin{theorem}\label{thm:composition}
There is a natural complex of abelian groups
\begin{multline*}
H_n(X^n, X(n)^b_c)\otimes H_n(X^n, X(n)^a_b)\to 
 \substack{\bigoplus\\ \mu+\nu=n}  \:  H_{\mu}(X^{\mu}, X(\mu)^b_c)\otimes H_\nu (X^\nu, X(\nu)^a_b)\to\\
 \to \substack{\bigoplus\\ \mu+\nu=n-1}  \: H_{\mu}(X^{\mu}, X(\mu)^b_c)\otimes H_\nu (X^\nu, X(\nu)^a_b)
\end{multline*}
for which  the  kernel $K$  of the second map comes with a natural map    $K\to H_n(X^{n}, X(n)^a_c)$.
It has the property that  the resulting map 
\[
H_n(X^n, X(n)^b_c)\otimes H_n(X^n, X(n)^a_b)\to H_n(X^{\mu}, X(\mu)^a_c
\]
represents the path composition map \eqref{eqn:comp1}.

In case $a=b=c$, the summands associated with $\nu=0$ and $\mu=0$ are trivial (they involve the trivial $H_0 (X^0, X(0)^a_a)$ as tensor factor) and this makes the above map factor through 
$H_{n-1} (X^{\mu}, X(\mu)^b_c)\otimes H_{n-1}(X^\nu, X(\nu)^a_b)$. This reproduces the map \eqref{eqn:comp2}. 
\end{theorem}

The rest of this section is devoted to the proof of Theorem \ref{thm:composition}. Since we already treated  the case $n=1$, we assume in what follows that $n>1$.

\begin{lemma}\label{lemma:zero} When $\mu+\nu={n-1}$, the inclusion  
\[
j: (X^{\mu}, X(\mu)^b_c)\times \{b\}\times (X^\nu, X(\nu)^a_b\subset (X^n, X(n)^a_c)
\]
 induces the zero map on $H_n$.
\end{lemma}

The proof of Lemma \ref {lemma:zero}  requires another lemma.

\begin{lemma}\label{lemma:motivic_fg}
The pair $ (X^n, X(n)^a)$ is $(n-1)$-connected and the composite $ H_1(X)^{\otimes n}\to H_n(X^n)\to H_n (X^n, X(n)^a)$  is an isomorphism.
\end{lemma}
\begin{proof}
We prove this with induction on $n$. This begins trivially for $n=1$ and we therefore assume $n>1$.
Consider the projection 
\[
f: (X^n,X(n)^a)\to X^{n-1}, (x_n, \dots, x_1)\mapsto (x_{n-1}, \dots, x_1) 
\]
The fiber over a point $(x_{n-1}, \dots, x_1)\in X^{n-1}$ is a copy of $(X, x_{n-1})$ unless this point lies in 
$X(n-1)^a$, in which case the fiber is a copy of $(X,X)$. By induction $(X^{n-1},X(n-1)^a)$ is $(n-2)$-connected and 
 $ H_1(X)^{\otimes (n-1)}\to H_{n-1} (X^{n-1}, X(n-1)^a)$  is an isomorphism. On the other hand, the fiber pairs 
 $(X, x_{n-1})$ are $0$-connected and $H_1(X)\to H_1(X, x_{n-1})$ is an isomorphism. This implies that $(X^n,X(n)^a)$ is 
 $(n-1)$-connected and that (by a Leray-Hirch argument) the natural map
\[
H_1(X)^{\otimes n}\cong H_1(X)\otimes H_{n-1} (X^{n-1}, X(n-1)^a)\to H_n(X^n)\to H_n (X^n, X(n)^a)
\]
 is an isomorphism. The lemma follows.
\end{proof}

\begin{proof}[Proof of lemma \ref{lemma:zero}.]
The argument that follows is symmetric in $\nu$ and $\mu$ and this allows us to assume that  $\nu>0$.
Then  $j$ factors as
\[
(X^{\mu}, X(\mu)^b_c)\times\{b\}\times (X^\nu, X(\nu)^a_b)\subset (X^\mu,X(\mu)^b_c)\times (X^{1+\nu},X(1+\nu)^a)\subset (X^n, X(n)^a_c)
\]
By  Theorem \ref{thm:motivic}  the pair $(X^{\mu}, X(\mu)^b_c)$ is  $(\mu-1)$-connected and by    Lemma \ref{lemma:motivic_fg},  the pair $(X^{1+\nu},X(1+\nu)^a)$ is  $\nu$-connected. Since $n=\mu+\nu+1$, the K\"unneth theorem implies that
\[
H_n(X^\mu,X(\mu)^b_c)\times (X^{1+\nu},X(1+\nu)^a)\cong H_\mu(X^\mu,X(\mu)^b_c)\otimes H_{1+\nu}(X^{1+\nu},X(1+\nu)^a).
\]
It therefore suffices to show that $\{b\}\times (X^\nu, X(\nu)^a_b)\subset (X^{1+\nu},X(1+\nu)^a)$ induces the zero map on $H_{1+\nu}$. This follows from Lemma \ref{lemma:motivic_fg}, which states that
\[
H_{1}(X)^{\otimes(1+\nu)}\cong H_{1+\nu}((X,a)^{1+\nu})\to H_{1+\nu}(X^{1+\nu},X(1+\nu)^a)
\]
is an isomorphism, and so by fixing $x_{1+\nu}=b$ we get zero. 
\end{proof}

Recall that if $\g$ is a path in $X$ from $a$ to $b$ and $\delta$ a path from $b$ to $ c$, then its composite $\delta\gamma$ is the path in $X$ that assigns to $t\in[0,\half]$ the point $\g(2t)$ and to $t\in[\half,1]$ the point $\delta(2t-1)$. This suggests we  decompose $\triangle^n$ into pieces that are defined  by the closed  half spaces $t_\nu \le \half$ and  $t_\nu \ge \half$ and their nonempty intersections. Such an intersection  determines and is determined  by a pair of nonnegative integers $(\nu, \mu)$ with $\nu +\mu\le n$: it is then given as
\begin{equation}\label{eqn:simplexpiece}
1\ge t_n \ge\cdots \ge t_{n+1-\mu} \ge \half=t_{n-\mu}=\cdots=t_{\nu+1}\ge  t_\nu \ge\cdots \ge t_1\ge 0.
\end{equation}
We identify this intersection in the evident manner with $\triangle^\mu\times\triangle^\nu$. The  restriction of  $\delta\gamma$ to it is then 
$(\delta^{\mu}|\triangle^{\mu})\times (\g^\nu| \triangle^\nu)$ (see Figure \ref{fig:dec}) and  defines 
\[
\triangle^{\mu}[\delta]\otimes \triangle^\nu[\g^{\nu}]\in H_{\mu}(X^{\mu}, X(\mu)^b_c)\otimes H_\nu(X^\nu, X(\nu)^a_b),
\]
where we recall that we agreed to interpret $H_0(X^0, X(0)^p_q)$ as $0$ or $\ZZ$ according to whether  or not $p=q$.

This leads us  to introduce an auxiliary space that is defined as follows. For nonnegative integers $\nu, \mu$ with $\nu+\mu\le n$, put   
\[
{}^\mu X^\nu\{b\}:=X^\mu\times\{\underbrace{(b,b,\dots, b)}_{n-\mu-\nu}\}\times X^\nu
\]
(so when $\nu+\mu=n$, this is just a copy of $X^n$). 
It is clear that if  $\nu\ge \nu'>0$ and $\mu\ge \mu'>0$, then ${}^{\mu'} X^{\nu'}\{b\}\subset {}^\mu X^\nu\{b\}$. We glue the ${}^\mu X^\nu\{b\}$  along these inclusions, schematically

$\dots$\begin{tikzcd}[column sep=small]
  &X^n&                                              & X^n&                                             & X^n            \\
X^{n-1}\arrow[ur, "j_{3}", hook]   &     &  X^{n-1}\arrow[ur, "j_2", hook]\arrow[ul, "j_2" ', hook] &         & X^{n-1}\arrow[ul, "j_{1}" ', hook]\arrow[ur,"j_{1}", hook] &       
\end{tikzcd} \\
where $j_\nu $ inserts $b$ in the slot indexed by $\nu$, and denote the result $Z$. So $Z$ is a connected union of $n+1$ copies of $X^n$  
and there is an obvious  projection $Z\to X^n$ of `degree' $n+1$. 
Our main reason for introducing this space  $Z$ is that the map  $(\delta\gamma)^n|\triangle^n$ naturally lifts to $Z$ by letting it on the piece defined by \eqref{eqn:simplexpiece} take its values in the image of  ${}^\mu X^\nu\{b\}$. 

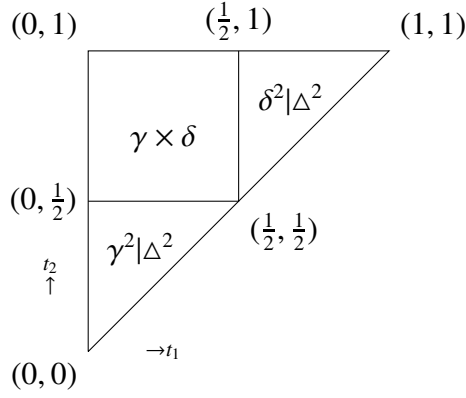
\begin{figure}\label{fig:dec}
\begin{centering}
\begin{tikzpicture}
\draw (0,0) --(0,4) -- (4,4)--(0,0);  
\draw (0,2) --(2,2) -- (2,4);
\node [above] at (1,2.5) {$\gamma\times\delta$};
\node [above] at (0.7,1) {$\gamma^2|\triangle^2$};
\node [above] at (2.7,3) {$\delta^2|\triangle^2$};
\node [below left] at (0,0) {$(0,0)$};
\node [left] at (0,2) {$(0,\half)$};
\node [above left] at (0,4) {$(0,1)$};
\node [above] at (2,4) {$(\half,1)$};
\node [above right] at (4,4) {$(1,1)$};
\node [below right] at (2,2) {$(\half,\half)$};
\node  at  (1,0) {$\substack{\rightarrow t_1}$};
\node  at  (-0.5,1) {$\substack{t_2\\\uparrow}$};
\end{tikzpicture}
\end{centering}
\caption{How $\delta *\gamma|\triangle^2$ is  described  as a sum of three (rescaled) maps.}
\end{figure}

We define a subset ${}^\mu X^\nu\{b\}^a_c\subset {}^\mu X^\nu\{b\}$ by
\[
{}^\mu X^\nu\{b\}^a_c:=\big(X^\mu\times\{b\}^{n-\mu-\nu}\times X(\nu)^a\big) \cup \big(X(\mu)_c\times\{b\}^{n-\mu-\nu}\times X^\nu\big)
\]
(so for example, ${}^0X^{n}\{b\}^a_c=X(n)^a$). In other words, the pair $( {}^\mu X^\nu\{b\},{}^\mu X^\nu\{b\}^a_c)$ 
decomposes as the product 
\[
( {}^\mu X^\nu\{b\},{}^\mu X^\nu\{b\}^a_c):=(X^\mu, X(\mu)_c)\times\{\underbrace{(b,b,\dots, b)}_{n-\mu-\nu}\}\times (X^\nu, X(\nu)^a).
\]
We put  $Z^a_c:=\cup_{\mu, \nu}{}^\mu X^\nu\{b\}^a_c\subset Z$.
Since the projection $Z\to X^n$ makes   ${}^\mu X^\nu\{b\}^a_c$ land in $X(n)^a_c$,  we have  a map of pairs 
\[
pr: (Z, Z^a_c)\to (X^n,X(n)^a_c).
\] 

\begin{proof}[Proof of Theorem \ref{thm:composition}]
Let  $Z'\subset Z$ be the union of $Z^a_c$ and the locus where distinct copies of $X^n$ meet,  so
$Z'= Z^a_c\cup \cup_{\mu+\nu=n-1} X^\mu \times\{b\}\times X^\nu$. 
This means that when $\mu+\nu=n$, the pair $({}^{\mu} X^\nu\{b\} ,Z'\cap{}^{\mu} X^\nu\{b\})$  decomposes as $(X^{\mu}, X(\mu)^b_c)\times (X^\nu, X(\nu)^a_b)$, so that we have we a relative homeomorphism  
\[
\textstyle \coprod_{\mu+\nu={n}}(X^{\mu}, X(\mu)^b_c)\times (X^\nu, X(\nu)^a_b)\to  (Z, Z'), \\
\]
We  find in the same manner a relative homeomorphism  
\[
\textstyle \coprod_{\mu+\nu={n-1}}(X^{\mu}, X(\mu)^b_c)\times (X^\nu, X(\nu)^a_b)\to  (Z', Z^a_c)
\]
(here we use that $n>1$).
Both  induce isomorphisms on relative homology by excision.  By Theorem \ref{thm:motivic} the  left hand side is $(n-1)$-connected resp.\ $(n-2)$-connected and hence so is the right hand side. 
By feeding  this in the following fragment the long exact homology sequence of the triple $(Z,Z', Z^a_c)$, 
\[
H_n(Z',Z^a_c)\to H_n(Z,Z^a_c)\to H_n(Z,Z')\to H_{n-1}(Z', Z^a_c), 
\]
we get the exact sequence that appears in the statement of Theorem \ref{thm:composition}:
\begin{multline}\label{eqn:exactfragment}
\oplus_{\mu+\nu={n-1}} H_n\big((X^{\mu}, X(\mu)^b_c)\times (X^\nu, X(\nu)^a_b\big)\to H_n(Z,Z^a_c)\to\\
\to  \oplus_{\mu+\nu=n} H_\mu (X^{\mu}, X(\mu)^b_c)\otimes H_\nu(X^\nu, X(\nu)^a_b)\to\\
\to \oplus_{\mu+\nu={n-1}} H_\mu (X^{\mu}, X(\mu)^b_c)\otimes H_\nu(X^\nu, X(\nu)^a_b)
\end{multline}
Recall that  $K$ stands for the kernel of the last arrow of the exact sequence \eqref{eqn:exactfragment}. So this is also the cokernel of the first arrow: we  have an  exact sequence 
\[
\oplus_{\mu+\nu={n-1}} H_n\big((X^{\mu}, X(\mu)^b_c)\times (X^\nu, X(\nu)^a_b\big)\to H_n(Z,Z^a_c)\to  K\to 0.
\]
Lemma \ref{lemma:zero}  tells us that the map $pr_*: H_n(Z,Z^a_c)\to H_n(X^n, X(n)^a_c)$ has the property that its precomposite with the first map of \eqref{eqn:exactfragment} is zero. It follows that  we have an induced map $K\to  H_n(X^n, X(n)^a_c)$. To complete the  proof of Theorem \ref{thm:composition} it remains to define the map 
\[
H_n (X^{\mu}, X(\mu)^b_c)\otimes H_n(X^\nu, X(\nu)^a_b)\to K.
\] 
If  $\alpha\in H_n(X^n,X(n)^a_b)$ and $\beta\in H_n(X^n,X(n)^b_c)$,  then the  truncation maps determine 
an image of $\beta\otimes\alpha$ in
$\oplus_{\mu+\nu=n} H_\mu (X^{\mu}, X(\mu)^b_c)\otimes H_\nu(X^\nu, X(\nu)^a_b)$. This image   lies in $K$, because these truncation maps are compatible. We  have thus obtained  a well-defined map 
\[
H_n (X^{\mu}, X(\mu)^b_c)\otimes H_n(X^\nu, X(\nu)^a_b)\to H_n(X^n, X(n)^a_c).
\] 
The construction makes it clear that this  is the map  defined by path composition.
\end{proof}

\begin{example}\label{example:}
The multiplication  $\Ical_a/\Ical_a^2\otimes \Ical_a/\Ical_a^2\to \Ical^2_a/\Ical_a^3$ is given by the composite 
of K\"unneth multiplication $H_1(X)\otimes H_1(X)\to H_2(X^2)$ and the natural map
$H_2(X^2)\to H_2(X^2, X(2)^a_a)$. The latter is easily verified to be the cokernel of the map $H_2(X)\to H_1(X)\otimes H_1(X)$ defined by the diagonal embedding (and that is dual to the cup product).
\end{example}

\begin{remark}\label{rem:}
In case $X$ is a complex algebraic variety, the identification of $\ZZ\pi_X(a,b)_n$ with $H_n(X^n, X(n)^a_b)$ 
(or $\ZZ\oplus H_n(X^n, X(n)^a_a)$ when $a=b$) can be used to define a mixed Hodge structure on 
$\ZZ\pi_X(a,b)_n$. The geometric characterization of its structural maps (truncation, composition and the coproduct, defined when $a=b$) shows that these are morphisms of mixed Hodge structures. This also shows that any  such structure on $\ZZ\pi_X(a,b)_n$  which is `sufficiently natural' must be this one.
Indeed, we thus  reproduce the mixed Hodge structure on these Hopf algebras and  modules that was defined by Hain (see \cite{hain:pi1}, \cite{hz}). 
\end{remark}

\end{document}